\documentclass[12pt]{article}   	
\usepackage{geometry}                		
\usepackage[parfill]{parskip}    		
\usepackage{graphicx}				
\usepackage[hyperindex=true,pagebackref,colorlinks=true,pdfpagemode=none,urlcolor=blue,linkcolor=blue,citecolor=blue,pdfstartview=FitH]{hyperref}
\usepackage{amsmath,amsfonts}
\usepackage{color}
\usepackage{amsthm}
\usepackage{boolexpr}
\usepackage{amssymb,latexsym,pstricks,pst-plot}
\usepackage[english]{babel}
\usepackage{pgf}
\usepackage{tikz}
\usetikzlibrary{arrows,automata}
\usepackage[latin1]{inputenc}
\usepackage[UKenglish]{isodate}
\allowdisplaybreaks

\usepackage{amssymb}
\newtheorem{theorem}{Theorem}[section]
\newtheorem{lemma}[theorem]{Lemma}

\newtheorem{corollary}[theorem]{Corollary}

\theoremstyle{definition}

\newtheorem*{remark}{Remarks}

\title{Factorials and Legendre's three-square theorem}
\author{Rob Burns}

\begin{document}
\maketitle
\begin{abstract}
We provide a necessary and sufficient condition for $n!$ to be a sum of three squares. The condition is based on the binary representation of $n$ and can be expressed by the operation of an automaton.
\end{abstract}

\section{Introduction}
\label{intro}

For any positive integer $n > 1$, we know that $n!$ cannot be a perfect square since, by Bertrand's Postulate, there is a prime $p$ between $n/2$ and $n$. The highest power of $p$ dividing $n!$ must then be $1$. We also know that, apart from a few exceptions, $n!$ cannot be written as the sum of two squares. This is a consequence of two results. The first result is the well known \emph{Sum of two squares theorem} \cite[Thm~366]{hw2008} which states:

\bigskip

\begin{theorem}[Sum of two squares theorem] 
\label{sumoftwosquares}
An integer greater than one can be written as a sum of two squares if and only if its prime decomposition contains no term $p^k$, where $p$ is a prime with $p \equiv 3 \pmod 4$ and $k$ is an odd number.
\end{theorem}

\bigskip

The second result was proved by Erd{\"o}s in 1935.\cite{Erd_s_1935}

\bigskip

\begin{theorem}[Erd{\"o}s] 
\label{Erdos}
If $n$ is a positive integer $\geq 7$ then there is a prime $p$ of the form $p \equiv 3 \pmod 4$ with $n/2 < p \leq n$.
\end{theorem}

Using theorem \ref{sumoftwosquares} and the same reasoning as for Bertrand's Postulate, Erd{\"o}s concluded that the only factorials that can be written as a sum of two squares are $1!$, $2!$ and $6!$.

On the other hand, Lagrange's four-square theorem, states that every natural number, including factorials,  can be represented as the sum of four integer squares.\cite[Thm~369]{hw2008}

The purpose of this paper is to provide a necessary and sufficient condition for a factorial be written as the sum of three squares. Entry A084953 in The Online Encyclopaedia of Integer Sequences includes a list of those $n!$ which cannot be written as a sum of three squares.\cite{oeis} There is no obvious pattern in the list. In the next section we will derive a condition that determines when $n!$ can be written as a sum of three squares. The last section of the paper describes an automaton that takes the binary representation of $n$ as input and decides whether $n!$ can be written as the sum of three squares.

\section{Writing factorials as the sum of three squares}
\label{main}
Our starting point is Legendre's three square theorem.\cite[Thm~9.8]{Takloo_Bighash_2018}

\bigskip

\begin{theorem}[Sum of three squares theorem] 
\label{sumofthreesquares}
A positive integer can be represented as the sum of three squares of integers if and only if it is not of the form $4^a(8b+7)$ for integers $a, b \geq 0$.
\end{theorem}

\bigskip

Any integer can be written uniquely in the form 
\begin{equation}
\label{rep}
2^{\gamma}Z  \,\,\,\, \text{where } \, Z \pmod{8} \in \{1, 3, 5, 7 \}. 
\end{equation}

In the case of $n!$, the value of $\gamma$ is given by Legendre's formula.

\bigskip

\begin{theorem}[Legendre's formula] 
\label{Legendreformula}
Let $n$ be a positive integer with binary representation $n = \sum_{k \geq 0} a_k 2^k$. Then the highest power of $2$ dividing $n!$ is $n - \sum_{k \geq 0} a_k$.
\end{theorem}

\bigskip

We will introduce some notation to assist in the calculation of the value of $Z \pmod{8}$ in (\ref{rep}).

Let $n \in \mathbb{N}$, with binary representation given by $n = \sum_{k \geq 0} a_k 2^k$, where all but finitely many $a_i$ are zero. For $i \in \{3,5,7 \}$, define $\alpha_i = \alpha_i(n)$ by

\begin{align}
\label{alpha3}
\alpha_3 = \alpha_3(n) &:= \#\Big\{ k \geq 0: \sum_{i=k}^{k+2} a_i 2^{i-k} \in \{3,4 \}  \Big\} \\
\label{alpha5}
\alpha_5 = \alpha_5(n) &:= \#\Big\{ k \geq 0: \sum_{i=k}^{k+2} a_i 2^{i-k} \in \{5,6 \} \Big\} \\
\label{alpha7}
\alpha_7 = \alpha_7(n) &:= \#\Big\{ k \geq 0: \sum_{i=k}^{k+2} a_i 2^{i-k} = 7 \Big\}.
\end{align}

For $n, x \in \mathbb{N}$ we also define:

$$
A(n, x) := \max_k \{k: 2^k x \leq n \} + 1.
$$

\bigskip

Finally, for $i \in \{1,3,5,7 \}$ and $k \in \mathbb{N}$ we define $A_{i,k}(n)$ and $A_i(n)$ by
\begin{align*}
A_{i,k}(n) :=& \,\, \# \{ x: x \equiv i \pmod{8}, \,\, 2^k x \, \leq \, n \} \\
A_i(n) :=& \sum_{x \, \equiv \, i \pmod{8} } A(n, x).
\end{align*}

\bigskip
We now give some technical lemmas that describe how the various definitions above connect with each other. 
\bigskip

\begin{lemma}
\label{Ais}
For each $i \in \{1,3,5,7 \}$ we have:
$$
A_i(n) = \sum_{k \geq 0} A_{i,k}(n)
$$
\end{lemma}
\begin{proof}
Fix $i$. For each $x \equiv i \pmod{8}$, if $A(n, x) = m$, then $x$ is also counted in $m$ of the sets $\{ x: x \equiv i \pmod{8}, \,\, 2^k x \leq n \}$.
\end{proof}
\bigskip
\begin{lemma}
\label{mod2}
Let $n \in \mathbb{N}$ with binary representation given by $n = \sum_{k \geq 0} a_k 2^k$ where all but finitely many $a_i$ are zero. Then,
$$
A_{i,k}(n)  \pmod{2} \, \equiv \,
\begin{cases}
    a_{k+3} , & \text{if } \, \, \sum_{j = k}^{k+2} 2^{j-k} a_j \, < \, i \\
    a_{k+3} + 1, & \text{if } \, \, \sum_{j = k}^{k+2} 2^{j-k} a_j \, \geq \. i.
\end{cases}
$$
\end{lemma}
\begin{proof}
In general, for fixed $v, w \in  \mathbb{N}$ with $0 \leq w < 8$, 
$$
\# \{x: x \leq 8v + w: x \equiv i \pmod{8} \, \} \, = \,
\begin{cases}
    v , & \text{if } \, \, w \, < \, i \\
    v + 1, & \text{if } \, \, w \, \geq \, i.
\end{cases}
$$
Taking into account the binary representation of $\lfloor n/2^k \rfloor$, we then have,
$$
A_{i,k}(n)  =
\begin{cases}
    \sum_{j \geq k+3} 2^{j-k-3} a_j, & \text{if } \, \, \sum_{j = k}^{k+2} 2^{j-k} a_j < i \\
    \sum_{j \geq k+3} 2^{j-k-3} a_j + 1, & \text{if } \, \, \sum_{j = k}^{k+2} 2^{j-k} a_j \geq i.
\end{cases}
$$
The lemma follows by evaluating this expression modulo $2$.
\end{proof}

\bigskip

\begin{corollary}
\label{cor}
$A_i (n)  \pmod{2} \equiv \sum_{k \geq 0} a_{k+3} + \# \{ k: k \geq 0: \sum_{j = k}^{k+2} 2^{j-k} a_j \geq i \}$.
\end{corollary}

\bigskip

We now have our main result.

\bigskip

\begin{theorem}
\label{mainthm}
Let $n \in \mathbb{N}$ with binary representation given by $n = \sum_{k \geq 0} a_k 2^k$, where all but finitely many $a_i$ are zero. If $\gamma$ is the highest power of $2$ dividing $n!$, then $n! = 2^{\gamma}Z$, where $Z$ satisfies 
\begin{align*}
Z \equiv 3^{\alpha_3(n)}(-1)^{\alpha_5(n)} \pmod{8}
\end{align*}
\end{theorem}
\begin{proof}
Separating $n!$ into odd and even factors, we have
\begin{align*}
n! \,\, =& \,\, 2^{\gamma} \prod_{x \,\, odd \,  ; x \, \leq \, n} x^{A(n,x)} \\
     =& \,\, 2^{\gamma} \prod_{i \, \in \, \{1,3,5,7 \} } Z_i
\end{align*}

where,

\begin{align*}
Z_i = \prod_{x \, \equiv \, i  \pmod{8}} x^{A(n,x)} 
\end{align*}

We are interested in $Z_i  \pmod{8}$. By lemma \ref{Ais},

\begin{align*}
Z_i  \pmod{8} =& i^{\sum_{x \, \equiv \, i \pmod{8}} A(n,x)} \\
                      =&  i^{A_i (n)}
\end{align*}

Since $i^2 \equiv 1 \pmod{8}$, we have, by Corollary \ref{cor},
$$
Z_i  \pmod{8} = i^{\sum_{k \geq 0} a_{k+3} \, + \, \# \{ k: k \geq 0: \, \sum_{j = k}^{k+2} 2^{j-k} a_j \, \geq \, i \}}.
$$
Putting everything together, and using the definitions of $\alpha_3$, $\alpha_5$ and $\alpha_7$ in (\ref{alpha3}), (\ref{alpha5}), (\ref{alpha7}), we have, $n! = 2^{\gamma}Z$, where 

\begin{align*}
Z \pmod{8} \equiv & \prod_{i \in \{1, 3, 5, 7 \}} Z_i \pmod{8} \\
                           =& \,\, \prod_{i} i^{\sum_{k \geq 0} a_{k+3} \, + \,  \# \{ k: \, k \, \geq \, 0: \sum_{j = k}^{k+2} 2^{j-k} a_j \, \geq \, i \, \}} \\
                           =& \,\, (3 \times 5 \times 7)^{\sum_{k \geq 0} a_{k+3}} \times 3^{\alpha_3 + \alpha_5 + \alpha_7} \times 5^{\alpha_5 + \alpha_7} \times 7^{\alpha_7} \\
                           =& \,\, 3^{\alpha_3} \times (3 \times 5)^{\alpha_5} \\
                          = & \,\, 3^{\alpha_3} \times (-1)^{\alpha_5}.
\end{align*}
\end{proof}
\bigskip

\begin{corollary}
If $n \in \mathbb{N}$, then $n!$ cannot be written as a sum of three squares if and only if $\gamma$ and $\alpha_3$ are even and $\alpha_5$ is odd.
\end{corollary}
\begin{proof}
Fix $n$, let $\gamma$ denote the highest power of $2$ dividing $n!$ and write $\bar{\gamma} = \gamma \pmod{2}$. Then, the result from Theorem \ref{mainthm} can be rewritten as 
\begin{equation}
\label{4xZ}
n! = 4^{x} Z, \,\, \text{ where } \,\, x = (\gamma - \bar{\gamma})/2 
\end{equation}
\bigskip

and $Z$ satisfies 

\bigskip

\begin{equation}
\label{Z}
Z \equiv 2^{\bar{\gamma}} \, 3^{\alpha_3} \, (-1)^{\alpha_5} \pmod{8}.
\end{equation}

\bigskip

The corollary follows from the identities $3^2 \equiv (-1)^2 \equiv 1 \pmod{8}$ and theorem \ref{sumofthreesquares}.
\end{proof}

\bigskip

\begin{remark}
We provide here a few applications of Theorem \ref{mainthm}. 
\bigskip

If $n = 2^k + w$ where $k \geq 5$ and $0 \leq w < 8$, then, from (\ref{Z}), $n! = 4^{x} Z$ where,

$$
Z  \pmod{8} \equiv
\begin{cases}
   1, & \text{if } \, \, w \in \{3, 4 \} \\
    2, & \text{if } \, \, w = 7 \\
    3, & \text{if } \, \, w = 2 \\
    5, & \text{if } \, \, w = 5 \\
    6, & \text{if } \, \, w \in \{ 0, 1, 6 \}.
\end{cases}
$$

\bigskip

If $n = \frac{2}{3} \, (16^k - 1)$ for $k \geq 1$, then $n!$ cannot be written as a sum of three squares. The binary representation of $n$ is $(10101010 ... 1010)_2 = (1010)_2 \times \sum_{i = 0}^k 16^i$, so $\gamma$ is even, $\alpha_3 = 0$ and $\alpha_5$ is odd. Similarly, if $n = \frac{4}{5} \, (16^{2k+1} - 1)$ for $k \geq 0$, then $n!$ cannot be written as a sum of three squares. Numbers of this form can be written as $(1100)_2 \times \sum_0^{2k} 16^i$. 

If $n$ is divisible by $4$ and $n!$ cannot be written as a sum of three squares, then neither can $(2n)!$. Adding a $0$ as the least significant digit of such an $n$ has no effect on the values of $\gamma$, $\alpha_3$ and $\alpha_5$. A more general statement is that, when $n$ is divisible by $4$, the value of $Z \pmod{8}$ in (\ref{4xZ}) is the same for both $n$ and $2n$.

Let $n = \sum_{k = 0}^r a_k 2^k$ and $m = n + 36 \times 2^{r+1}$. If $n!$ cannot be written as a sum of three squares then neither can $m!$. The binary representation of $m$ is obtained by adding $(100100)_2$ to the most significant end of the binary representation of $n$. Again, a more general statement is that the value of $Z \pmod{8}$ in (\ref{4xZ}) is the same for both $n$ and $m$.

\bigskip

Equation (\ref{Z}) suggests that the probability that $n!$ cannot be written as a sum of three squares is $1/8$. This is supported by numerical results. A proof would require that the values of $\gamma$, $\alpha_3$ and $\alpha_5 \pmod{2}$ are independent in a suitable sense. More generally, (\ref{Z}) can be used to derive heuristic estimates for the expected asymptotic values of $Z  \pmod{8}$. Table \ref{tablez} provides these estimates and compares them to the actual proportions for $n \leq 1000000$.

\bigskip

\begin{table}[htbp]
  \centering
   \begin{tabular}{ | c | c | p{3cm} |}
 \hline
$Z \pmod{8}$   & Actual & Estimate \\  \hline
1 & 0.124967 & 0.125  \\
2 & 0.249445 & 0.25 \\ 
3 & 0.124968 & 0.125 \\ 
5 & 0.125032 & 0.125 \\
6 & 0.250556 & 0.25 \\
7 & 0.125032 & 0.125 \\ \hline
  \end{tabular}
  \caption{Table of actual and estimated values of $Z \pmod{8}$ for $n \leq 1,000,000$.}
  \label{tablez}
\end{table}

\bigskip 

\end{remark}

\bigskip

\section{Building the automaton}
\label{automaton}

Let $n \in \mathbb{N}$ with $n! = 4^x Z$ as in (\ref{4xZ}). We will build an automaton which takes the binary digits of $n = \sum_{k \geq 0} a_k 2^k$ as input, starting with the least significant, $a_0$, and determines the value of $Z \pmod{8}$. In particular, the automaton determines whether $n!$ can be written as the sum of three squares. The automaton is constructed as the product of three separate automata, which keep track of the parity of $\gamma$, $\alpha_3$ and $\alpha_5$.

\subsection{Automaton for the parity of $\gamma$}
The automaton which calculates the parity of $\gamma$ is straightforward. It consists of two states which we will call $g_0$ and $g_1$. The state $g_0$ indicates that $\gamma$ is even and $g_1$ indicates that $\gamma$ is odd. From Theorem \ref{Legendreformula}, 
$$
\gamma \pmod{2} = (n - \sum_{k \geq 0} a_k) \pmod{2} = \sum_{k \geq 1} a_k \pmod{2}.
$$
The automaton enters state $g_0$ after the least significant binary digit is received as $\gamma = 0$ in both of the possible cases. After this, the parity of $\gamma$ is changed if the new digit is a $1$ and unchanged if the new digit is a $0$. The automaton for the parity of $\gamma$ is pictured in figure \ref{autog}.

\bigskip

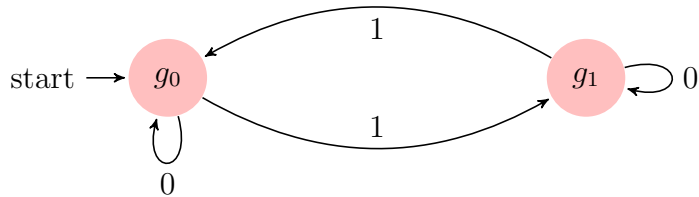
\begin{figure}[htbp]
\centering
\begin{tikzpicture}[->,>=stealth',shorten >=1pt,auto,node distance=5.5cm,
                    semithick]
  \tikzstyle{every state}=[fill=pink,draw=none,text=black]

  \node[initial,state] (A)                    {$g_0$};
  \node[state]         (B) [right of=A] {$g_1$};
  
  \path (A) edge [bend right] node {1} (B)
  (B) edge [bend right] node {1} (A)
  (A) edge [loop below] node {$0$} (A)
        (B) edge [loop right] node {$0$} (B);                      
 \end{tikzpicture}
                       \caption{Automaton for the parity of $\gamma$}
                        \label{autog}
\end{figure}

\bigskip

\subsection{Automaton for the parity of $\alpha_3$}

The automaton which keeps track of the parity of $\alpha_3$ consists of 8 states. We call the states $s_{x, y, z}$ where $x, y, z$ are binary digits. The binary digit $x$ is $0$ when $\alpha_3$ is even and $1$ when $\alpha_3$ is odd. The digit $y$ is the most recent binary digit to have been read in from $n$ and $z$ is the second most recent binary digit. In describing the action of the automaton, we will ignore the technicalities of what happens when the first two digits of $n$ are received. Adding a $0$ as the most significant digit of $n$ does not change the value of $\alpha_3$. So when the automaton receives a $0$ the state $s_{x, y, z}$ transitions to $s_{x, 0, y}$. In the language of automata, this is described as:
$$
(\, s_{x, y, z}, 0) \,\,  \to \, \, s_{x, 0, y} .
$$
\bigskip
The state transitions when the automaton receives a $1$ are not quite as simple because each new calculation of $\alpha_3$ has to take into account the previous calculation. The state transitions are as follows:

\begin{align*}
(\, s_{0, 0, 0}, 1) \,\,  & \,\, \to \,\, \,\, s_{1, 1, 0} \\
(\, s_{0, 0, 1}, 1) \,\,  & \,\, \to \,\, \,\, s_{0, 1, 0} \\
(\, s_{0, 1, 0}, 1) \,\,  & \,\, \to \,\, \,\, s_{1, 1, 1} \\
(\, s_{0, 1, 1}, 1) \,\,  & \,\, \to \,\, \,\, s_{0, 1, 1} \\
(\, s_{1, 0, 0}, 1) \,\,  & \,\, \to \,\, \,\, s_{0, 1, 0} \\
(\, s_{1, 0, 1}, 1) \,\,  & \,\, \to \,\, \,\, s_{1, 1, 0} \\
(\, s_{1, 1, 0}, 1) \,\,  & \,\, \to \,\, \,\, s_{0, 1, 1} \\
(\, s_{1, 1, 1}, 1) \,\,  & \,\, \to \,\, \,\, s_{1, 1, 1} . 
\end{align*}

\bigskip

\subsection{Automaton for the parity of $\alpha_5$}

The automaton which keeps track of the parity of $\alpha_5$ also consists of 8 states. We call the states $t_{x, y, z}$ where $x, y, z$ are binary digits. The binary digit $x$ is $0$ when $\alpha_5$ is even and $1$ when $\alpha_5$ is odd. The digit $y$ is the most recent binary digit to have been read in from $n$ and $z$ is the second most recent binary digit. The other comments from the previous section apply here as well. When a $0$ is received, the transition is described by

$$
(\, t_{x, y, z}, 0) \,\,  \to \, \, t_{x, 0, y} .
$$

The transitions when a $1$ is received are described by:

\begin{align*}
(\, t_{0, 0, 0}, 1) \,\,  & \,\, \to \,\, \,\, t_{0, 1, 0} \\
(\, t_{0, 0, 1}, 1) \,\,  & \,\, \to \,\, \,\, t_{1, 1, 0} \\
(\, t_{0, 1, 0}, 1) \,\,  & \,\, \to \,\, \,\, t_{1, 1, 1} \\
(\, t_{0, 1, 1}, 1) \,\,  & \,\, \to \,\, \,\, t_{0, 1, 1} \\
(\, t_{1, 0, 0}, 1) \,\,  & \,\, \to \,\, \,\, t_{1, 1, 0} \\
(\, t_{1, 0, 1}, 1) \,\,  & \,\, \to \,\, \,\, t_{0, 1, 0} \\
(\, t_{1, 1, 0}, 1) \,\,  & \,\, \to \,\, \,\, t_{0, 1, 1} \\
(\, t_{1, 1, 1}, 1) \,\,  & \,\, \to \,\, \,\, t_{1, 1, 1} . 
\end{align*}

\bigskip

\bibliographystyle{plain}
\begin{small}
\bibliography{Factorial}
\end{small}

\end{document}